\documentclass[12pt]{amsart} 

\usepackage[utf8]{inputenc}
\usepackage[english]{babel}
\usepackage{amsmath,amssymb,amsthm,mathrsfs}
\usepackage{esint}
\usepackage[colorlinks,linktocpage]{hyperref}
\usepackage[all]{xy}
\usepackage{import}
\usepackage{pdfpages}
\usepackage{transparent}
\usepackage{xcolor}
\usepackage{comment}

\usepackage{verbatim, latexsym, amssymb, amsmath,color, mathrsfs}
\usepackage{enumitem}
\usepackage{MnSymbol}
\usepackage{epsfig}
\usepackage{xcolor}

\usepackage{stmaryrd}
\usepackage{float}

\usepackage{geometry}
\usepackage{hyperref}

\newtheorem{thm}{Theorem}[section]
\newtheorem{lem}[thm]{Lemma}
\newtheorem{cor}[thm]{Corollary}
\newtheorem{prop}[thm]{Proposition}

\theoremstyle{definition}
\newtheorem{defn}{Definition}[section]
\newtheorem{prob}{Problem}[section]

\theoremstyle{remark}
\newtheorem{rmk}{Remark}[section]
\newtheorem*{rmk*}{Remark}

\newtheorem*{fact*}{Fact}

\def\dV{\mathrm{dvol}}

\title{On $3$-manifolds with small mass and $L^2$-curvature}
\author{Conghan Dong}
\address{Department of Mathematics, Duke University, 120 Science Dr, Durham, NC 27710, USA}
\email{conghan.dong@duke.edu}

\author{Antoine Song}
\address{California Institute of Technology\\ 177 Linde Hall, \#1200 E. California Blvd., Pasadena, CA 91125}
\email{aysong@caltech.edu}

\begin{document}
\date{\today}

\maketitle

\begin{abstract}
	One of S.T. Yau's problems asks the following: given a $3$-dimensional asymptotically flat manifold $M$ with non-negative scalar curvature and $L^2$-norm of the curvature tensor at most $1$, if the mass of $M$ is small, is there a bilipschitz diffeomorphism from $M$ to the flat Euclidean space $\mathbb{R}^3$? We provide a strong positive answer to this problem by using our previous work \cite{DS25}. 
\end{abstract}

\section{Introduction}

The famous Positive Mass Theorem, first proved by Schoen and Yau \cite{SchoenYau79}, states that if $(M,g)$ is an asymptotically flat 3-manifold with non-negative scalar curvature, then the  ADM mass of any given end is non-negative, and it is zero if and only if $(M, g)$ is isometric to the Euclidean 3-space $\mathbb{R}^3$. For related work, see for instance the references in \cite{DS25}.
In an attempt to better understand the ``almost-rigidity'' or ``stability'' properties of this rigidity theorem when the mass is close to zero, S.T. Yau asked the following problem in \cite{Yau93}:
\begin{prob}(\cite[Problem 17]{Yau93})\label{yau-prob}
	Let $M$ be a complete three-dimensional asymptotically flat manifold with non-negative scalar curvature. If the total mass is small when we normalize the $L^2$-norm of the curvature tensor to be one, is $M$ diffeomorphic to $\mathbb{R}^3$ so that the metric is also uniformly equivalent to the flat metric?
\end{prob}

In other words, Problem \ref{yau-prob} asks whether imposing a small upper bound on the scale invariant quantity $m(g) \int_{M} |\mathrm{Rm}_g|^2 \mathrm{dvol}_g$ implies that $(M,g)$ is bilipschitz to the Euclidean 3-space $\mathbb{R}^3$. Later, G. Huisken and T. Ilmanen formulated a related stability conjecture \cite[Section 9]{HI01}, where only a small upper bound on $m(g)$ is assumed. Building on the harmonic map method \cite{BKKS22}, we confirmed Huisken-Ilmanen's conjecture in \cite{DS25}, where we also resolved the related Bartnik's strict positivity conjecture \cite{Bartnik89}. 
Several other versions of the stability problem have been explored, for instance from  a differential geometric perspective \cite{BF99,FK01,Corvino05,Finster07,Lee09,KKL21, ABK22, Dong22} or from a more geometric measure theoretic viewpoint \cite{Sormani23,LS14,Sormani16,BKS21, HL15,HLS17,HLP22,AP20}.
Closely related stability problems for tori with almost nonnegative scalar curvature in a pointwise or integral sense were investigated in 
\cite{LNN20,Allen21,ABK22,CL22,BC25}. See also other references in \cite{DS25, song2025entropy, dong2025stability}.

Problem \ref{yau-prob} is guided by the general phenomenon of $\epsilon$-regularity in geometric analysis. $\epsilon$-regularity theorems and more generally regularity improvement phenomena involving integral curvature bounds play a central role in geometric analysis, see \cite{BKN89, Anderson90, Gao90, PetersenWei97, CheegerTian06, CWY22, Che22} for results with a small $L^p$-curvature assumption. In spite of their apparent simplicity, $L^p$-curvatures $[\int_{M} |\mathrm{Rm}_g|^p \mathrm{dvol}_g]^{\frac{1}{p}}$ are not well-understood quantities, see questions in M. Berger's book \cite[Section 11.3]{Ber03} (in particular the discussion after \cite[Lemma 268]{Ber03}), and the conjecture in \cite[Introduction]{Son21} for Einstein 4-manifolds.

Our goal is to prove the following theorem, which provides a positive and quantitative answer to Problem \ref{yau-prob}: 

\begin{thm}\label{main-thm}
For any $\epsilon >0$, there exists a constant $\epsilon _0(\epsilon )>0$ depending only on $\epsilon$ such that the following holds.
	Let $(M, g)$ be a complete $3$-dimensional asymptotically flat manifold with non-negative scalar curvature, whose $L^2$-norm of curvature satisfies $\int_{M} |\mathrm{Rm}_g|^2 \mathrm{dvol}_g \leq 1 $. 
	 If the ADM mass of one end of $M$ satisfies $m(g) \leq \epsilon _0(\epsilon )$, then $(M,g)$ is $\epsilon$-bilipschitz to the flat Euclidean 3-space $(\mathbb{R}^3,g_{\mathrm{Eucl}})$, namely there exists a diffeomorphism $\phi :  M\to \mathbb{R}^3$ such that
	\begin{align*}
		(1-\epsilon ) g \leq \phi ^* g_{\mathrm{Eucl}} \leq (1+\epsilon ) g.
	\end{align*} 

\end{thm}

\begin{rmk} Actually, our proof implies that the conclusion of Theorem \ref{main-thm} holds under a substantially more general assumption: for any $\epsilon>0$ and any $p>\frac{3}{2}$, if 
the scaling invariant quantity $m(g)^{2p-3} \int_{M}|\mathrm{Rm}_g|^p \dV_g$ is smaller than some positive constant $\epsilon_0(\epsilon,p)$ only depending on $\epsilon$ and $p$, then $(M,g)$ is $\epsilon$-bilipschitz to the flat Euclidean 3-space $(\mathbb{R}^3,g_{\mathrm{Eucl}})$. Note that $\frac{3}{2}$ is a natural threshold here, since $\int_{M}|\mathrm{Rm}_g|^\frac{3}{2} \dV_g$ is scaling invariant and does not involve the mass. 
\end{rmk}

The proof of Theorem \ref{main-thm} is based on our recent work \cite{DS25}, where we proved an almost rigidity result for the positive mass theorem without assuming an $L^2$-bound on the curvature.
The idea can be summarized as follows. 
Let $0<\epsilon _0 \ll 1$ be a fixed small constant. By a simple rescaling argument, it is enough to show the theorem when $(M,g)$ satisfies
\begin{align*}
	m(g) \leq \epsilon _0,\ \int_{M} |\mathrm{Rm}_{g}|^2 \mathrm{dvol}_g \leq  \epsilon _0.
\end{align*} 
For a uniformly small $\epsilon _0$, by the work of D. Yang \cite{Yang92a, Yang92b, Yang92} on $3$-manifolds with small $L^2$-curvature, we know that if the volume is locally uniformly non-collapsed, then there is a uniform lower bound on the $W^{2,2}$-harmonic radius and thus on the $C^\alpha$-harmonic radius for some $\alpha \in (0,\frac{1}{2})$. On the other hand, by applying 
ideas from our proof of Bartnik's strict positivity conjecture \cite[Theorem 1.4]{DS25},
we prove that if the mass is small enough and if the metric on a ball $B$ in $(M,g)$ is uniformly controlled in the $C^0$-topology, then $B$ has an almost Euclidean volume. In other words, an a priori coarse volume lower bound can be improved to an almost Euclidean volume bound. This way, using the asymptotic flatness of $(M,g)$ and a continuity argument, we deduce that the almost Euclidean volume bound holds everywhere, which already yields local $C^{\alpha }$-closeness of the metric $g$ to a Euclidean metric. For the last step, we apply  \cite{DS25} again to construct a global bilipschitz map from $(M,g)$ to $\mathbb{R}^3$.

\subsection*{Acknowledgements}
We would like to thank Eric Chen for discussions. C. D. would like to thank Hubert Bray for his encouragement. 
A. S. was partially supported by NSF grant DMS-2405175 and an Alfred P. Sloan Research Fellowship.

\section{Preliminaries}

\subsection{Metrics with bounded $L^p$-curvature}
We collect some results about the regularity of metrics under bounded $L^p$-curvature assumptions that we will use later. Let us first recall the definition of harmonic radius.
\begin{defn} \label{def:harmonic rad}
	Let $(M^n, g)$ be an $n$-dimensional Riemannian manifold. Given $p> \frac{n}{2}$, the $W^{2,p}$-harmonic radius at $x \in M$ is the largest number $r_{H}(x) = r_{H}(p,x)$ such that on the geodesic ball $B=B_g(x, r_{H}(x))$, there is a harmonic coordinate chart $(u^1,\ldots, u^{n}): B \to \mathbb{R}^n$, such that for $g_{ij} = g\left( \frac{\partial }{\partial u^i}, \frac{\partial }{\partial u^j} \right) $, 
	\begin{align*}
		&\frac{1}{2} \left( \delta _{ij} \right) \leq \left( g_{ij} \right) \leq 2 \left( \delta _{ij} \right) ,\\
		&r_{H} ^{1- \frac{n}{p}} \| \partial g_{ij} \|_{L^p} + r_{H}^{2- \frac{n}{p}} \|\partial ^2 g_{ij}\|_{L^p} \leq 1.
	\end{align*} 
\end{defn}
\begin{rmk} \label{remharmrad}
	By the Sobolev embedding theorem, a lower bound $r_H(x)\geq r_0>0$ on the $W^{2,p}$-harmonic radius at $x\in (M,g)$ yields a lower bound $r_H'(x)\geq C(r_0)>0$ on the analogously defined $C^{\alpha }$-harmonic radius $ r_H'(x)$ for  $\alpha \in ( 0, 2- \frac{n}{p}) $, where $C(r_0)$ only depends on $r_0$. That in turn gives a bound depending only on $r_0$ on the (relative) isoperimetric constant and the Sobolev constant of the ball $B_g(x,C(r_0))$ (see e.g. \cite{Yang92a} for definitions).
\end{rmk}

Let $\omega_n$ be the volume of the Euclidean unit $n$-ball.

\begin{thm} [{\cite{Yang92a,Yang92b}}]\label{yang-harmonic}
    Let $n \geq 3, p> \frac{n}{2}$, $0< \eta< 1$, and $(M^n, g)$ be a complete Riemannian manifold. There exists a constant $\epsilon(n, p, \eta)$ such that if a geodesic ball $B= B_g(x, 1)$ satisfies 
    \begin{align*}
    \mathrm{Vol}_g(B ) &\geq \eta ^n \omega_n ,\\
        \| \mathrm{Rm}_g\|_{L^p(B)} &\leq \epsilon(n,p, \eta),
    \end{align*}
    then there exists a uniform $r_0>0$, such that the $W^{2,p}$-harmonic radius at $x$ satisfies  $r_H(x)\geq r_0$.
\end{thm}

\begin{proof}

Thanks to \cite[Theorem 7.4]{Yang92a}, we know that the isoperimetric constant on $B$ is uniformly controlled from below (see \cite[Section 4]{Yang92a} for definitions). It is well-known that this constant is directly related to the Sobolev constant 
\cite[(4.1)]{Yang92a} so the Sobolev constant on $B$ is uniformly bounded from above. By
\cite[Theorem 7.1]{Yang92b}, we conclude the proof.
\end{proof}

Under integral control of the negative part of the Ricci curvature, the following relative volume comparison was shown in \cite{PetersenWei97}.

\begin{thm}[{\cite[Theorem 1.1]{PetersenWei97}}]\label{vol-comp}
	Let $n\geq 3$, $p> \frac{n}{2}$, and $(M^n,g)$ be a complete Riemannian manifold. There is a constant $C(n, p, R)$ which is nondecreasing in $R$ such that for any $x \in M $, when $r<R$, we have
	\begin{align*}
		\left( \frac{\mathrm{Vol}_g(B_g(x,R) )}{\omega _n R ^{n}}  \right) ^{\frac{1}{2p}} - \left( \frac{\mathrm{Vol}_g(B_g(x,r) )}{\omega _n r ^{n}} \right) ^{\frac{1}{2p}} \leq C(n, p, R) \cdot \left( \int_{M} |(\mathrm{Ric}_g)_{-}|^{p} \mathrm{dvol}_g \right) ^{\frac{1}{2p}}.
	\end{align*} 
	In particular, for $0<r\leq 1$, we have
\begin{align*}
	\left( \frac{\mathrm{Vol}_g(B_g(x,r) )}{\omega _n r ^{n}}  \right) ^{\frac{1}{2p}} \leq 1+ C(n, p) \left( \int_{M} |(\mathrm{Ric}_g)_{-}|^{p} \mathrm{dvol}_g \right) ^{\frac{1}{2p}}.
\end{align*} 
\end{thm}

We recall several definitions from \cite{Yang92} for later use. Let $C_{S}(\Omega )$ denote the Sobolev constant, that is, the smallest $A>0$ such that
	\begin{align*}
	\left( \int_{\Omega } |f| ^{\frac{2n}{n-2}}\right) ^{\frac{n-2}{n}} \leq A \int_{\Omega }|\nabla f|^2,\quad \forall f \in C ^{\infty}_{c}(\Omega )
	\end{align*} 
	holds. We say that an $(A,B)$-Sobolev inequality holds on $(\Omega , g)$ if
	\begin{align*}
	\left( \int_{\Omega } |f| ^{\frac{2n}{n-2}}\right) ^{\frac{n}{n-2}} \leq A \int_{\Omega }|\nabla f|^2 + B \int_{\Omega } f ^2,\ \forall f \in C^{\infty}_{c}(\Omega ).
	\end{align*} 
    Given $0< \delta_0 < 1$, the weak injectivity radius $\rho (\delta_0 ,x)$ at $x$ is defined by the largest radius $\rho $ such that 
	\begin{align*}
		C_{S}(B_g(x, \rho ) ) \leq \delta_0 ^{-2} C_{S}(\mathbb{R}^n),
	\end{align*} 
	and for any $B_g(y, r) \subset B_g(x, \rho )$,
	\begin{align*}
		\mathrm{Vol}_g(B_g(y,r) ) \leq \delta_0 ^{-n} \omega_n r ^{n}.
	\end{align*} 
	  Define
	\begin{align*}
		\rho (\delta_0 ,M) = \inf_{x \in M} \rho (\delta_0 ,x).
	\end{align*}     

\begin{rmk}\label{harmtoinj}
    From the definition of harmonic radius and Remark \ref{remharmrad}, a lower bound $r_H(x)\geq r_0>0$ on the  $W^{2,p}$-harmonic radius at $x$ with $p> \frac{n}{2}$ implies a  lower bound $\rho(\delta_0, x)\geq C(r_0)>0$ on the weak injectivity radius $\rho(\delta_0, x)$ for some $\delta_0>C'(r_0)>0$ where the constants $C(r_0),C'(r_0)$ only depend on $r_0$.
\end{rmk}
\begin{rmk}\label{yang1}
    By \cite[Lemma 3.1]{Yang92}, a uniform lower bound on the weak injectivity radius $\rho(\delta_0, M)$ implies a uniform $(A, B)$-Sobolev inequality, with $A, B$ depending only on $n, \delta_0$ and $\rho(\delta_0, M).$
\end{rmk}

We also recall a result related to smoothings of metrics by local Ricci flows \cite{Yang92a, Yang92}. Consider
a smooth $n$-dimensional manifold $M$ with a complete Riemannian metric $g_0$ and $\Omega $ an open subset of $M$. Let $\chi \in C^{\infty}_{c}(\Omega)$ be a compactly supported nonnegative smooth function on $\Omega$. The local Ricci flow is defined by the following evolution equation:
\begin{align}\label{local-RF-eq}
	\frac{\partial g(t)}{\partial t} = - 2 \chi ^2 \mathrm{Ric}_{g(t)},\ g(0)=g_0.
\end{align}
\begin{thm}[{\cite[Theorems 9.8, B.1, 9.9]{Yang92a}, \cite[Theorem 4.1]{Yang92}}] \label{yang2}
	Let $q>n$. Assume that a uniform $(A_0, B_0)$-Sobolev inequality holds on $(\Omega, g_0)$ and that
\begin{align*}
	\left( \int_{\Omega} |\mathrm{Rm}_{g_0}| ^{\frac{q}{2}} \mathrm{dvol}_{g_0} \right) ^{\frac{2}{q}} \leq \mu_0.
\end{align*} 
Then the local Ricci flow (\ref{local-RF-eq})
has a unique smooth solution for $0 \leq t \leq T_0$, where
$$T_0 ^{-1}\leq C(n,q) (\|\nabla \chi\|^2_{\infty}+ A_0^{-1}B_0 + A_0^{n/(q-n)}\mu_0^{q/(q-n)}).$$
Moreover, for $0<t\leq T_0$, we have the following uniform estimates
\begin{align*}
	\|\mathrm{Rm}_{g(t)}\|_{L ^{q/2}( \Omega) } &\leq 2 \|\mathrm{Rm}_{g_0}\|_{L ^{q/2}(\Omega) },\\
	\chi (x) |\mathrm{Rm}_{g(t) }|(x) &\leq C(n, q) A_0 ^{n/q} t ^{- n/q} \| \chi ^{2(1- n/q)}\mathrm{Rm}_{g_0}\|_{L^{q/2}(\Omega)}.
\end{align*} 
If the stronger assumption $\|\mathrm{Rm}_{g_0}\|_{\infty, \Omega}\leq K$ holds, then for some $0<T_1 \leq T_0$ with $T_1^{-1} \leq C(n)(\|\nabla \chi\|_{\infty} + K)$, we have $\|\mathrm{Rm}_{g(t)}\|_{\infty, \Omega} \leq 2K$ for $0\leq t\leq T_1$.
\end{thm}

\subsection{Small ADM mass and harmonic maps}\label{section:ADM-mass}

A smooth orientable connected complete Riemannian $3$-manifold $(M, g)$ is called asymptotically flat if there exists a compact subset $K \subset M$ such that $M\setminus K= \bigsqcup_{k=1}^{N} M_{\mathrm{end}}^k$ consists of finite pairwise disjoint ends, and for each $1\leq k \leq N$, there exist $B>0, \sigma > \frac{1}{2}$, and a $C^\infty$-diffeomorphism $\Phi_k : M_{\mathrm{end}}^k \to  \mathbb{R}^3 \setminus B_{\mathrm{Eucl}}(0, 1)$ such that under this identification, 
	$$
	| \partial ^{l}(g_{ij}- \delta _{ij})(x)| \leq B|x|^{-\sigma - |l|},
	$$ for all multi-indices $|l|=0,1,2$ and any $x \in \mathbb{R}^3\setminus B_{\mathrm{Eucl}}(0,1)$, where $B_{\mathrm{Eucl}}(0,1)$ is the standard Euclidean ball with center $0$ and radius $1$. Furthermore, we always assume the scalar curvature $R_g$ is integrable over $(M,g)$. For a given end, its ADM mass from general relativity is then well-defined (see \cite{ADM61, Bartnik86}) and given by
	$$m(g):= \lim_{r\to \infty} \frac{1}{16\pi} \int_{S_r} \sum_{i,j} (g_{ij,i}-g_{ii,j}) v^j dA$$
	where $v$ is the unit outer normal to the coordinate sphere $S_r$ of radius $|x|=r$ in the given end, and $dA$ is its area element.

Assume that $(M^3, g)$ is a complete asymptotically flat $3$-manifold with non-negative scalar curvature, which satisfies for a given end:
$$ m(g) \leq \epsilon _0.$$
Recall that for that end, 
there is an associated ``exterior region'' $M_{ext}$, which contains that end, is diffeomorphic to $\mathbb{R}^3$ minus finitely many disjoint balls, and has a (possibly empty) minimal boundary, but contains no other closed minimal surface in this region (see e.g. \cite{HI01}). 
Given any Riemannian metric $h$, let $d_h$ denote the induced Riemannian distance.

Building on the harmonic map method of \cite{BKKS22}, the following was shown in \cite[Proposition 4.13 and (27)]{DS25}:

\begin{thm}\label{ds}
Given a fixed $0<\epsilon < 1$, for all small enough $\epsilon _0$, if $m(g)\leq \epsilon_0$ then there exists an unbounded connected good region $E \subset M_{ext}$ which is closed and has smooth boundary, so that $E$ contains the end of $M$,
\begin{align*}
\mathrm{Area}_g(\partial E) \leq C\cdot \epsilon _0 ^{2- \epsilon },
\end{align*}
and for any base point $q \in E$, there is a harmonic map $\mathbf{u}: M_{ext} \to \mathbb{R}^3$ satisfying
\begin{itemize}
	\item [(1)] $\mathbf{u}|_{E}$ is a diffeomorphism onto its image and $\mathbf{u}(q) =0 \in \mathbb{R}^3$ ;
	\item [(2)] the Jacobian matrix of $\mathbf{u}$ satisfies $|\mathrm{Jac} (\mathbf{u}) - \mathrm{Id}| \leq \epsilon $ on $E$ ;
	\item [(3)] for any $D>0$, for any $x, y \in \hat{B}_{g, E}(q, D)$,
		\begin{align*}
			|\hat{d}_{g,E}(x,y) - d_{g_{\mathrm{Eucl}}}(\mathbf{u}(x), \mathbf{u}(y) )| \leq C(D)\cdot \epsilon, 
		\end{align*} 
        where the constant $C(D)$ only depends on $D$;
	\item[(4)]
$
			(\mathbf{u})_{\sharp} \left( \mathrm{dvol}_{g}|_{\hat{B}_{g, E}(q,D)} \right)$ converges to $ \mathrm{dvol}_{g_{\mathrm{Eucl}}}|_{B_{g_{\mathrm{Eucl}}}(0,D)}$ weakly as measures as $\epsilon _0 \to 0$.

\end{itemize}
Here, ${g_{\mathrm{Eucl}}}$ is the standard flat metric on $\mathbb{R}^3$, and the geometric quantities with the hat $\,\hat{}\,$ notation are taken with respect to the induced length metric in $E$ (i.e. the distance between two points is measured using the paths inside $E$).
\end{thm}

\section{Local non-collapsing}

In this section, we will consider small enough $\epsilon_0>0$ and assume that $(M, g)$ is a complete asymptotically flat $3$-manifold with non-negative scalar curvature, which satisfies for a given end:
\begin{align}\label{epsilon-condition}
	m(g) \leq \epsilon _0, \quad \int_{M} |\mathrm{Rm}_g|^2 \mathrm{dvol}_{g} \leq \epsilon _0.
\end{align} 
Let 
\begin{align}\label{goodregion}
E\subset M
\end{align} 
be the good region inside $M$ whose existence and properties are guaranteed by Theorem \ref{ds}. We will use the notation $C$ to denote some universal constant (which may change from line to line).

First we observe the following lemma.
\begin{lem}\label{lem-harmonic}
	For any $v_0 \in (0,1)$, there exists uniform $ \epsilon _0(v_0)>0$ such that for all $0< \epsilon _0 \leq \epsilon _0(v_0)$, for any $x \in M$ with $\mathrm{Vol}_g(B_g(x, 1) ) \geq v_0$, there exists a uniform $r_0 = r_0(v_0)>0$ so that the $W^{2,2}$-harmonic radius satisfies
	\begin{align*}
		r_{H}(y) \geq r_0,\quad \forall y \in B_g(x,2).
	\end{align*} 

\end{lem}

\begin{proof}
	By the volume comparison result Theorem \ref{vol-comp}, for any $y \in B_g(x, 3)$, we have
\begin{align*}
	\left( \frac{\mathrm{Vol}_g(B_g(y ,1) )}{ \omega _3 } \right) ^{\frac{1}{4}} 
	&\geq \left( \frac{\mathrm{Vol}_g(B_g(y, 4 ) )}{ \omega _3 4 ^3 } \right)^{\frac{1}{4}} - C\cdot  \epsilon _0 ^{\frac{1}{4}} \\
	&= \left( 64 ^{-1} \omega _3 ^{-1} v_0 \right) ^{\frac{1}{4}} - C\cdot  \epsilon _0 ^{\frac{1}{4}} \\
	&\geq \left( 10 ^{-3} v_0\right) ^{\frac{1}{4}},
\end{align*} 
where in the last inequality, we used $\epsilon _0\leq \epsilon(v_0)\ll v_0$.
Applying Theorem \ref{yang-harmonic} together with the small $L^2$-curvature assumption (\ref{epsilon-condition}), the conclusion follows if $\epsilon_0$ is uniformly small.

\end{proof}

Next, we prove the following key proposition:

\begin{prop}\label{vol-improve-prop}
	For any $\epsilon >0$, there exists $\epsilon _0(\epsilon )>0$ such that for all $\epsilon _0 \leq \epsilon _0(\epsilon )$, for any $x \in M$, if $\mathrm{Vol}_g( E \cap B_g(x, 1) ) \geq (1- \epsilon ) \omega _3$, then $\mathrm{Vol}_g(E \cap B_g(x, 1) ) \geq (1- \frac{1}{2} \epsilon ) \omega _3$.
\end{prop}

\begin{proof}
	We argue by contradiction. Assuming that the conclusion does not hold, there exists  a fixed $\epsilon>0$ and a sequence of pointed asymptotically flat $3$-manifolds $(M_i, g_i, p_i)$  satisfying (\ref{epsilon-condition}) with $ \epsilon _i \to 0$ and for the corresponding good regions $E_i$ in $M_i$ given by Theorem \ref{ds}, we have 
	\begin{align}\label{contra-assump}
		(1- \epsilon ) \omega _3 \leq \mathrm{Vol}_{g_i}(E_i \cap B_{g_i}(p_i, 1) ) < (1- \frac{1}{2} \epsilon ) \omega _3.
	\end{align} 
	By Lemma \ref{lem-harmonic}, we know $r_{H}(y) \geq 2r_0$ for a uniform $0<r_0 \ll 1$ and for all $y \in B_{g_i}(p_i, 2)$. By the Sobolev embedding theorem, up to a subsequence, we can assume that for some $\alpha \in (0, \frac{1}{2})$, for some $C^\alpha$-Riemannian metric $g_0$ on a manifold and a geodesic ball $B_{g_0}(p_0, 2)$ with respect to $g_0$,
	\begin{align*}
		(B_{g_i}(p_i, 2) , g_i) \to (B_{g_0}(p_0, 2), g_0)
	\end{align*} 
	in the $C^{\alpha }$-topology. In other words, there exist $C^{1, \alpha }$-embeddings 
	$$\psi _i: B_{g_0}(p_0, 2) \to B_{g_i}(p_i, 2+ r_0)$$
	with
    $$\quad  \psi _i(p_0) = p_i \quad  \text{and}\quad B_{g_i}(p_i, 2) \subset \psi _i(B_{g_0}(p_0, 2) ) \subset B_{g_i}(p_i, 2+ \delta _i)$$ for some positive $\delta _i \to 0$, and $\psi _i^* g_i \to g_0$ as tensors in the $ C^{\alpha }$-topology. In particular $\psi _i^* g_i \to g_0$ uniformly, so finite perimeter subsets of $B_{g_i}(p_i, 2)$ satisfy a relative isoperimetric inequality with a uniform constant (cf. \cite[Theorem 5.11 (ii)]{EG15}). 

Set $D _i:=B_{g_i}(p_i, 2) \cap  E_i$. If $\mathrm{Vol} _{g_i}(D _i) < \frac{1}{2} \mathrm{Vol}_{g_i}(B_{g_i}(p_i, 2) )$, by the relative isoperimetric inequality, we know that for all $i$ large enough,
$$\mathrm{Vol}_{g_i}(D_i) \leq C\cdot \mathrm{Area}_{g_i}(\partial E_i \cap B_{g_i}(p_i, 2)) ^{\frac{3}{2}} \leq C\cdot \epsilon _i ^{2} \to 0 \quad \text{as $i\to \infty$},$$ 
where in the last inequality, we used the area bound given in Theorem \ref{ds}. This
contradicts our noncollapsing assumption  (\ref{contra-assump}).  So by $\mathrm{Vol}_{g_i} (D_i) \geq  \frac{1}{2} \mathrm{Vol}(B_{g_i}(p_i, 2) )$, and by the relative isoperimetric inequality again, we have in fact
	\begin{align}\label{small-vol-bad}
		\lim_{i\to \infty}  \mathrm{Vol}_{g_i}(B_{g_i}(p_i, 2) \setminus E _i)= 0.
	\end{align} 
By the uniform convergence of the metric tensors, together with the contradiction assumption (\ref{contra-assump}),
\begin{align}\label{desired contra}
	\mathrm{Vol}_{g_0}(B_{g_0}(p_0, 1) ) = \lim_{i\to \infty} \mathrm{Vol}_{g_i}(B_{g_i}(p_i, 1)\cap E_i) \leq (1- \frac{1}{2} \epsilon ) \omega _3. 
\end{align}

\begin{lem}\label{lem-key-DS}
Let $\hat{d}_{g_i, E_i}$ be the induced length metric  given by the restriction of $g_i$ to $E_i$.
		For any $x \in B_{g_0}(p_0, 1) $, up to a subsequence, the maps
		\begin{align*}
			\psi_i ^{-1}: \left( E_i \cap B_{g_i}(x_i, r_0), \hat{d}_{g_i, E_i} \right) \to \left( B_{g_0}(x,  r_0), d_{g_0} \right) 
		\end{align*} 
		are $\delta' _i$-Gromov-Hausdorff approximation for $x_i:= \psi _i(x)$ and some positive $\delta' _i \to 0$. Moreover, $(\psi _i^{-1})_{\sharp} (\mathrm{dvol}_{g_i}|_{E_i \cap B_{g_i}(x_i,  r_0 )})$ weakly converges to $\mathrm{dvol}_{g_0}|_{B_{g_0}(x,  r_0)}$.
	\end{lem}
	\begin{proof}
		The proof is similar to \cite[Lemma 5.4]{DS25}. The map $\psi _i ^{-1}$ restricted to $E_i \cap B_{g_i}(x_i,  r_0)$ is a $C^{1,\alpha}$-diffeomorphism onto its image, so we can view $E_i \cap B_{g_i}(x_i,  r_0)$ as a subset of $B_{g_0}(x,  r_0)$ equipped with the metric $g_i$, and we will discard the notation $\psi _i$ in what follows. Earlier, we saw that $g_i \to g_0$ in the $C^{\alpha }$-topology, so we have
		\begin{align}\label{unif cv}
			(1- \delta _i) g_0 \leq g_i \leq (1+ \delta  _i) g_0
		\end{align} 
        for some positive $\delta_i\to 0$.
        Recall the notation $\hat{B}_{g_i, E_i}(y, r)$ 
        for the region in $E_i$ equal to the metric $r$-ball centered at $y$ in $E_i$ with respect to the induced length metric on $E_i$.
		By the smallness of the volume of the bad region  (\ref{small-vol-bad}) and the uniform convergence (\ref{unif cv}), we can choose basepoints $x_i '  \in E_i \cap B_{g_i}(x_i, r_0)$ so that for any fixed $r>0 $ and any $\hat{B}_{g_i, E_i}(y_i, r) \subset \hat{B}_{g_i, E_i}(x_i', \frac{3}{2}r_0)$ with $y_i \to y\in B_{g_0}(p_0, 2) $, the volume of good regions converges:
		\begin{align*}
			\mathrm{Vol}_{g_i}( \hat{B}_{g_i, E_i}(y_i, r) ) \to \mathrm{Vol} _{g_0} ( B_{g_0}(y, r) ). 
		\end{align*} 
        Notice that $\mathrm{Vol} _{g_0} ( B_{g_0}(y, r) )\geq c_{g_0} r ^{3} $
		for some constant $c_{g_0}>0$ depending only on $g_0$. This coarse volume lower bound and the same argument as in \cite[Lemma 5.4]{DS25} then implies the following properties for any fixed small $r>0$ and some sequence $\delta_i'\to 0$:
		\begin{itemize}
			\item [(i)] $B_{g_0}(x, r_0)$ is contained in the $\delta_i'$-neighborhood of $E_i \cap B_{g_i}(x_i,  r_0)$ in terms of the metric $g_0$;
			\item [(ii)] For any $y, z \in E_i \cap B_{g_i}(x_i, r_0)$ with $d_{g_0}(y, z) \geq r$,
				\begin{align*}
					| \hat{d}_{g_{i}, E_i}(y,z) - d_{g_0}(y, z)| \leq \delta_i';
				\end{align*} 
			\item [(iii)] $\mathrm{dvol}_{g_i}|_{E_i \cap B_{g_i}(x_i, r_0)} \to \mathrm{dvol}_{g_0}|_{B_{g_0}(x,  r_0)}$ weakly as measures.
		\end{itemize}
		Finally, by taking a sequence $r\to 0$ and choosing a subsequence if necessary, we conclude.

	\end{proof}

Let $\mathbf{u}_i:E_i\to \mathbb{R}^3$ be the harmonic maps given by Theorem \ref{ds} restricted to $E_i$. We can now apply the same Arzel\`a-Ascoli type argument as in the proof of \cite[Theorem 5.1]{DS25} to the domains $\Omega_i := B_{g_i}(p_i, 1)$, with (\ref{small-vol-bad}) and Lemma \ref{lem-key-DS} replacing respectively \cite[Lemma 5.3]{DS25} and \cite[Lemma 5.4]{DS25} (while in the context of \cite[Theorem 5.1]{DS25}, we considered a fixed domain $\Omega$ with a fixed metric, in our present case the metrics $g_i$ in $\Omega_i$ converge uniformly to a $C^\alpha$-metric $g_0$).
We conclude that $\mathbf{u}_i|_{ B_{g_i}(p_i, 1)\cap E_i}$ converges  in the uniform topology to an isometric embedding $\mathbf{u}_\infty: (B_{g_0}(p_0, 1) , g_0) \to \mathbb{R}^3$. This implies that $\mathrm{Vol}_{g_0}(B_{g_0}(p_0, 1) ) = \omega _3$, which is the desired contradiction with  (\ref{desired contra}).

\end{proof}

By a continuity argument, we obtain the following uniform noncollapsing result.
\begin{lem}\label{vol-lower-lem}
	For any $\epsilon >0$, there exists $\epsilon _0(\epsilon) $ such that whenever $\epsilon _0 \leq \epsilon _0(\epsilon )$, we have $\mathrm{Vol}_g(E \cap B_g(x, 1) ) \geq (1- \epsilon ) \omega _3$ for any $x \in M$. Moreover, $\mathrm{Vol}_g(B_g(x,1) \setminus E) \leq 2 \epsilon\cdot \omega_3.$
\end{lem}
\begin{proof}
Consider the set
\begin{align*}
	\mathcal{A}:= \{ x \in M: \mathrm{Vol}_g(E \cap B_g(x, 1) ) \geq (1- \epsilon ) \omega _3\} .
\end{align*}
By the asymptotic flatness condition on $M$, $\mathcal{A} \neq \emptyset$.
Since $E$ is a smooth domain, $\mathrm{Vol}_g(E \cap B_g(x,1) )$ is continuous with respect to $x$, so $\mathcal{A}$ is a closed subset. By Proposition \ref{vol-improve-prop}, for $\epsilon_0(\epsilon)$ small enough, $\mathcal{A}$ is also open. Thus $\mathcal{A}= M$. The upper bound for $\mathrm{Vol}_g(B_g(x,1) \setminus E)$ follows from (\ref{small-vol-bad}). Alternatively, one can also prove it by combining the lower volume bound with the volume comparison Theorem \ref{vol-comp}: for any $\epsilon_0$ smaller than a uniform constant,
\begin{align*}
    \mathrm{Vol}_g(B_g(x,1) \setminus E) &\leq \mathrm{Vol}_g(B_g(x,1) ) - \mathrm{Vol}_g(E\cap B_g(x,1) )\\
    &\leq (1+ C\cdot \epsilon_0^{1/4})^4 \omega_3 - (1- \epsilon) \omega_3\\
    &\leq 2 \epsilon \omega_3.
\end{align*}

\end{proof}

Together with Lemma \ref{lem-harmonic}, we immediately get

\begin{cor}\label{har-rad-prop}
There exists a uniform $0<r_1\leq 1$ such that for any complete asymptotically flat $3$-manifold $(M,g)$   with non-negative scalar curvature and (\ref{epsilon-condition}) satisfied for a uniformly small $\epsilon_0$, for any $x \in M$,
    the $W^{2,2}$-harmonic radius at $x$ satisfies $r_H(x)\geq r_1$.
\end{cor}

\section{Proof of the main theorem}

\begin{proof}[Proof of Theorem \ref{main-thm}]

If $(M_1, g_1)$ is a complete asymptotically flat $3$-manifold with non-negative scalar curvature, such that for a chosen end,
	\begin{align*}
		m(g_1) \leq \epsilon _0,\quad \int_{M} |\mathrm{Rm}_{g_1}|^2\mathrm{dvol}_{g_1} \leq 1 
\end{align*}
for some constant $\epsilon_0>0$, 
then the rescaled metric $\tilde{g}=\epsilon _0 ^{-1} g_1$ satisfies
\begin{align} \label{equiv}
	m(\tilde{g}) \leq \sqrt{ \epsilon _0}, \quad \int_{M}|\mathrm{Rm}_{\tilde{g}}|^2 \mathrm{dvol}_{\tilde{g}} \leq \sqrt{ \epsilon _0}.
\end{align} 
Hence, to show the main theorem, it is clearly enough to prove that for any $\epsilon>0$, the desired statement holds if the metric satisfies the upper bounds (\ref{equiv}) for some $\epsilon_0=\epsilon_0(\epsilon)$ small enough.

Thus, fix $\epsilon \in (0,1)$, let $\epsilon_0>0$ which will be fixed in the proof and which will only depend on $\epsilon$, and assume that $(M, g)$ is a complete asymptotically flat $3$-manifold with non-negative scalar curvature, which satisfies for some chosen end:
\begin{align}\label{epsilon-conditionbis}
	m(g) \leq \epsilon _0, \quad \int_{M} |\mathrm{Rm}_g|^2 \mathrm{dvol}_{g} \leq \epsilon _0.
\end{align} 
In the following, we denote by $C$ a universal positive constant, which may change from line to line.

Let us explain why we can assume without loss of generality that the metric $g$ is uniformly bounded in the $C^l$-topology for any integer $l\geq0$. 
We choose an exhaustion $\Omega_i $ of $(M,g)$ and choose smooth functions $\chi_i$ taking values in $[0,1]$ and supported in $\Omega_i$ such that $\chi_i \to 1$ and $\|\nabla \chi_i\|_{\infty} \to 0$ locally uniformly. By the uniform lower bound on the harmonic radius in Corollary \ref{har-rad-prop} and Remark \ref{harmtoinj}, we know that the weak injectivity radius satisfies a uniform lower bound $\rho (\delta_0 , M) \geq \rho _0$ for some uniform $0<\delta_0 < 1$ and $\rho _0>0$. By Remark \ref{yang1}, an $(A_0, B_0)$-Sobolev inequality holds on $M$ for uniform $A_0, B_0>0$. By Theorem \ref{yang2}, we have a sequence of local Ricci flows $g_i(t)$ satisfying $\frac{\partial g_i(t)}{\partial t} = -2 \chi_i^2 \mathrm{Ric}_{g_i(t) }$ and $g_i(0)= g$, all of which exist for a uniform time $T_0>0$. Since $(M,g)$ is asymptotically flat, in particular having bounded geometry depending on the metric $g$, by the curvature estimate in Theorem \ref{yang2}, we know that $g_i(t)$ has bounded curvature for a short time, depending on the initial metric $g$ while independent of $i$. By Hamilton's compactness theorem \cite{Hamilton95}, up to a subsequence, $g_i(t) \to g(t)$ locally smoothly for a Ricci flow $g(t)$, $0\leq t \leq T_0$ on $M$ with $g(0)=g$. By the standard uniqueness result for Ricci flow with bounded curvature, $g(t)$ agrees with the classical Ricci flow starting from $g$.
		By the uniform estimates in Theorem \ref{yang2}, we have for any $t\in [0,T_0]$:
\begin{align*}
	\int_{M} |\mathrm{Rm}_{g(t)}|^2(t) \mathrm{dvol}_{g(t)} &\leq C \cdot \int_{M} | \mathrm{Rm}_g|^2 \mathrm{dvol}_{g} \leq C \cdot \epsilon _0,\\
	\|\mathrm{Rm}_{g(t)}\|_{C^0(M)} &\leq C\cdot t ^{- \frac{3}{4}} \|\mathrm{Rm}_{g} \|_{L ^2(M)}.
\end{align*} 
Note that $\int_{0}^{T_0} t^{-\frac{3}{4}}dt <\infty$ so using the curvature estimate above and the Ricci flow equation, we also have $e ^{-C \sqrt{\epsilon _0}  t ^{\beta }} g \leq g(t) \leq e ^{C \sqrt{\epsilon _0}  t ^{\beta  } } g$ for some $\beta  \in (0, \frac{1}{4})$. 
	  Set $h := g(T_0)$. Then 
      \begin{align}\label{bbilip}
		e ^{-C \sqrt{\epsilon _0} } g \leq h \leq e ^{C \sqrt{\epsilon _0} } g ,\ \ \int_{M}|\mathrm{Rm}_{h}|^2 \leq C\cdot \epsilon _0.
	\end{align}	 
    By Shi's estimates \cite{Shi89}, we have for any integer $l\geq 0$:
	\begin{align*}
		|\nabla ^{l} \mathrm{Rm}_{h}| \leq C(l).
	\end{align*} 
Moreover, it is well-known that $h$ is asymptotically flat with the same mass $m(h) = m(g) \leq \epsilon _0$ and $R_{h} >0$ (cf. \cite[Section 2]{YuLi}). 
Since $g$ and $h$ are $\epsilon$-bilipschitz by (\ref{bbilip}) if $\epsilon_0$ is small enough, we only need to prove the theorem for $h$ instead of $g$. 
This finishes the explanation of why $g$ can assumed to be uniformly $C^l$-bounded for our purpose.

In the remaining of the proof, we will thus assume that $(M,g)$ is as in (\ref{epsilon-conditionbis}), and additionally that the metric $g$ is uniformly bounded in the $C^l$-topology for any integer $l\geq0$. 
 Our goal is to show that  $g$ is $\epsilon$-bilipschitz to the flat $\mathbb{R}^3$ if $\epsilon_0=\epsilon_0(\epsilon)$ is small enough.

    We first notice that $M = M_{ext}$. Otherwise, $\partial M_{ext} \neq \emptyset$ consists of minimal $2$-spheres $\Sigma$. By the Gauss-Codazzi equation and the uniform curvature bound, 
    \begin{align}\label{K-upper}
        K_{\Sigma } = \frac{1}{2} R_{g} - \mathrm{Ric}_{g}(\nu , \nu ) - \frac{1}{2} |A_\Sigma|^2 \leq C.
    \end{align} 
    By the famous Penrose inequality \cite{HI01, Bray01}, we know that
    $$\mathrm{Area}_g(\Sigma) \leq 16 \pi\cdot m(g)^2 \leq C \cdot \epsilon_0^2.$$
    By the Gauss-Bonnet theorem, and the fact that $\Sigma$ consists of $2$-spheres, together with (\ref{K-upper}), we obtain
	\begin{align*}
		4 \pi \leq \int_{\Sigma } K_{\Sigma } \leq C\cdot \mathrm{Area}_g(\Sigma ) \leq C\cdot \epsilon _0^2,
	\end{align*} 
a contradiction for all small enough $\epsilon _0$.

Recall by Theorem \ref{ds} that 
we have a good region
$$E\subset M$$
and a 
globally well-defined harmonic map 
$$\mathbf{u} = (u^1, u^2, u^3):M\to \mathbb{R}^3.$$
Let $k\in \{1,2,3\}$.
By the mass inequality (\cite[Theorem 1.2]{BKKS22}), we have 
\begin{align}\label{mass}
	\int_{M} \frac{|\nabla ^2 u^{k}|^2}{|\nabla u^{k}|} \leq C\cdot \epsilon _0.
\end{align}

Let $2r_0\in [0,1]$ be equal to the minimum of $1$ and the infimum of the $W^{2,2}$-harmonic radius on $(M,g)$ (see Definition \ref{def:harmonic rad}). By Corollary \ref{har-rad-prop}, $2r_0$ is lower bounded by a uniform strictly positive number independent of $(M,g)$ as long as $\epsilon_0$ is small enough.
For any $x \in M$, set $\tilde{u}^{k}:= u^k - \fint_{B_{g}(x, 2 r_0)} u^k$. Then using the uniform lower bound on $r_0$, we have in some harmonic coordinates in $B_{g}(x, 2 r_0)$:
\begin{align*}
	\Delta _{g} \tilde{u}^{k} = g ^{ ij} \partial_i \partial_j \tilde{u}^{k} = 0,
\end{align*}
and
$$ C^{-1} \delta _{ij} \leq g ^{ ij} \leq C \delta _{ij}\quad \text{and}\quad \| \partial ^{l} g ^{ ij}\|_{C^{\alpha }} \leq C(l),$$
where the constant $C(l)$ depends only on the integer $l\geq 0$. By the Poincar\'e inequality,
\begin{align*}
	\fint_{B_{g}(x, 2 r_0)} |\tilde{u}^{k}| \leq C r_0\cdot \fint_{B_g(x, 2 r_0)} |\nabla \tilde{u}^{k}| = C \fint_{B_g(x, 2 r_0)}|\nabla u^k|,
\end{align*} 
where we used in the last inequality that $r_0\leq 1$. By Moser's iteration argument, we have
\begin{align*}
	\| \tilde{u}^k \|_{L^{\infty}(B_{g}(x, \frac{7}{4} r_0) )} \leq C\cdot \fint_{B_{g}(x, 2 r_0)} |\nabla u^k|. 
\end{align*} 
By Schauder's estimate, we have
\begin{align}\label{schauder}
	\| \tilde{u}^{k}\|_{C^{2, \alpha }(B_g(x, \frac{3}{2}r_0) )} \leq C\cdot \fint_{B_{g}(x, 2 r_0)} |\nabla u^k|.
\end{align}

Since the Ricci curvature of $g$ is uniformly bounded thanks to the uniform $C^l$-bound on $g$, the Bishop inequality yields a uniform upper bound on the relative volume ratio. Together with the Poincar\'e inequality and the mass inequality (\ref{mass}), we have 
\begin{align*}
	&\left| \fint_{B_{g}(x,r_0)} |\nabla u^k| - \fint_{B_{g}(x, 2 r_0)} |\nabla u^k| \right| \\
	&\ \ \leq \frac{\mathrm{Vol}_g(B_g(x, 2r_0) )}{\mathrm{Vol}_g(B_g(x, r_0) )} \cdot \fint_{B_{g}(x, 2 r_0)} \left| |\nabla u^k| - \fint_{B_{g}(x, 2 r_0)} |\nabla u^k| \right|  \\
	&\ \ \leq C \cdot \fint_{B_{g}(x,2 r_0)} |\nabla |\nabla u^k|| \\
	&\ \ \leq C \left( \fint_{B_{g}(x, 2 r_0)} \frac{| \nabla |\nabla u^k| |^2}{|\nabla u^k|} \right) ^{\frac{1}{2}} \cdot \left( \fint_{B_g(x,2 r_0)} |\nabla u^k| \right)^{\frac{1}{2}} \\
	&\ \ \leq C\cdot  \epsilon _0 ^{\frac{1}{2}} \cdot  \left( \fint_{B_g(x, 2 r_0)} |\nabla u^k| \right)^{\frac{1}{2}},
\end{align*} 
which implies the inequality 
\begin{align}\label{ineqmin1}
	\fint_{B_{g}(x, 2 r_0)} |\nabla u^k| &\leq C\cdot \fint_{B_{g}(x, r_0)} (|\nabla u^{k}| +1). 
\end{align} 
Indeed, if $\fint_{B_g(x, 2 r_0)} |\nabla u^k| \leq 1$ then (\ref{ineqmin1}) is clearly true.
If $\fint_{B_g(x, 2 r_0)} |\nabla u^k| \geq 1$, then $\left( \fint_{B_g(x, 2 r_0)} |\nabla u^k| \right)^{\frac{1}{2}} \leq  \fint_{B_g(x, 2 r_0)} |\nabla u^k| $, which together with the above inequality implies that
\begin{align*}
	\fint_{B_g(x, 2 r_0)} |\nabla u^k| \leq \fint_{B_g(x,  r_0)} |\nabla u^k| + C \epsilon _0 ^{\frac{1}{2}} \fint_{B_g(x, 2 r_0)} |\nabla u^k|,
\end{align*} 
thus by taking $\epsilon _0 \ll 1$ so that $C \epsilon _0 ^{\frac{1}{2}} \leq \frac{1}{2}$, 
\begin{align*}
	\fint_{B_g(x, 2 r_0)} |\nabla u^k| \leq 2 \fint_{B_g(x, r_0)} |\nabla u^k|,
\end{align*} 
which gives (\ref{ineqmin1}) too.

Next, combining (\ref{schauder}) and (\ref{ineqmin1}), we have
\begin{align}\label{c2alpha}
	\|\tilde{u}^{k}\|_{C^{2, \alpha }(B_g(x, \frac{3}{2}r_0) )} \leq C\cdot \fint_{B_{g}(x, r_0)} (|\nabla u^{k}| +1). 
\end{align} 
We claim that there exists a small enough $\epsilon_0(\epsilon)$ depending only on $\epsilon$ such that for all $\epsilon_0 \leq \epsilon_0(\epsilon)$, the following holds:
\begin{align}\label{ifthen}
\text{if $\sup_{B_g(x, r_0)}|\nabla u^k| \leq 1+ \epsilon $, then $\sup_{B_g(x, r_0)}|\nabla u^k| \leq 1+ \frac{1}{2} \epsilon $}.
\end{align} 
To check this,  we notice that the assumption in (\ref{ifthen}) together with (\ref{c2alpha}) implies that
\begin{align*}
	\| \tilde{u}^{k}\|_{C^{2, \alpha }(B_g(x, \frac{3}{2}r_0) )} \leq C.
\end{align*} 
In particular,
\begin{align}\label{C0-3/2-bound}
    \|\nabla^2 u^k\|_{C^0(B_g(x, \frac{3}{2} r_0) )} = \|\nabla^2 \tilde{u}^k\|_{C^0(B_g(x, \frac{3}{2} r_0) )} \leq C.
\end{align}
Choose $0<\delta _2 \ll \delta _1 \ll \epsilon < 1$ and $\delta_1 < \frac{1}{2} r_0$. Recall that by Lemma \ref{vol-lower-lem}, for all small enough $\epsilon _0$, and all $x\in M$,
\begin{align*}
	\mathrm{Vol}_g(B_g(x, 2r_0 ) \setminus E) \leq C\cdot \delta_2 ,
\end{align*} 
where $E$ is the good region in $M$. 

Consider any point $z \in B_g(x, r_0) $. Then $B_g(z, \delta_1) \subset B_g(x, \frac{3}{2} r_0)$. Using the assumption in (\ref{ifthen}), (\ref{C0-3/2-bound}) and bullet (2) in Theorem \ref{ds}, we have for $\epsilon_0$ small enough:
\begin{align*}
	\fint_{B_{g}(z, \delta _1)} |\nabla u^k| &\leq \frac{1}{\mathrm{Vol}_g(B_g(z, \delta_1))} \int_{B_g(z, \delta _1) \cap E} |\nabla u^k| + C \delta _1 ^{-3} \cdot \mathrm{Vol}_g(B_g(z, \delta _1)\setminus E) \\
						&\leq 1+ \delta _2 + C\cdot \delta _1 ^{-3} \delta _2\\
						&\leq 1+ \delta_1.
\end{align*} 
So there exists $z' \in B_{g}(z, \delta _1)$ such that
$|\nabla u^{k}|(z') \leq 1+ \delta _1  $. By the uniform Hessian estimate (\ref{C0-3/2-bound}),  we have $|\nabla u^{k}|(z) \leq |\nabla u^{k}|(z') + C \cdot d(z, z') \leq 1+ \delta _1 + C \delta _1 \leq 1+ \frac{1}{2} \epsilon $, which implies the conclusion of the claim (\ref{ifthen}).


By this improvement property (\ref{ifthen}) and a continuity argument, we obtain that $|\nabla u^{k}|(x) \leq 1+ \epsilon $ holds for all $x \in M$. Besides, (\ref{C0-3/2-bound}) shows that $| \nabla ^2 u^{k}|(x) \leq C $ for all $x \in M$. Similarly, if $\epsilon_0$ is small enough, for all $x \in M$, using bullet (2) in Theorem \ref{ds} again, we have
\begin{align*}
	\fint_{B_g(x, \delta_1)} |\langle \nabla u^{i}, \nabla u^{j} \rangle_{g}- \delta _{ij}| &\leq \frac{1}{\mathrm{Vol}_g(B_g(x, \delta_1) )} \int_{B_g(x, \delta_1) \cap E} |\langle \nabla u^{i}, \nabla u^{j} \rangle_{g}- \delta _{ij}| \\
	 &\ \ + C\delta_1^{-3}\cdot \mathrm{Vol}_g( B_g(x, \delta_1) \setminus E) \\
&\leq \delta _1,
\end{align*}
which together with the uniform Hessian estimate implies that for all $x \in M$,
\begin{align} \label{biilii}
	|\langle \nabla u^{i}, \nabla u^{j} \rangle_{g}(x)- \delta _{ij}| \leq C \delta_1 \ll 1. 
\end{align} 
Thus $\mathbf{u}: M \to \mathbb{R}^3$ is globally nondegenerate.

Since $\mathbf{u}$ is one-to-one from the end of $M$ to the end of $\mathbb{R}^{3}$, by a degree argument,  $\mathbf{u}$ must be a diffeomorphism. Moreover, by (\ref{biilii}), $\mathbf{u}$ is $\epsilon$-bilipschitz whenever $\delta_1$ is small enough depending only on $\epsilon$.
We have thus proved that for any $\epsilon >0$, for all small enough $\epsilon _0 \leq \epsilon _0(\epsilon )$, any $(M,g)$ as in (\ref{epsilon-conditionbis}) is $\epsilon$-bilipschitz to the flat $\mathbb{R}^3$.


\end{proof}

\bibliographystyle{alpha}
\bibliography{main}

\end{document}